\documentclass[11pt]{amsproc}
\usepackage{hyperref} 
\usepackage[english]{babel}
\usepackage{geometry} 
\usepackage{mathrsfs}
\usepackage{amsrefs}
\usepackage{amssymb}
\usepackage{amsmath}
\usepackage{amsthm}
\DeclareMathOperator{\Lit}{Lit}
\newcommand{\ibet}{\eta}
\newcommand{\nop}[3]{\|#1\|_{{#2}\to{#3}}}
\newcommand{\one}{\mathbf{1}}
\newcommand{\CC}{\mathbf{C}}
\newcommand{\NN}{\mathbf{N}}
\newcommand{\RR}{\mathbf{R}}
\newcommand{\se}{\subseteq}

\newcommand{\diam}	{\operatorname{diam}} 
\newcommand{\girth}	{\operatorname{girth}} 
\newcommand{\rank}	{\operatorname{rank}} 
\newcommand{\Cay}	{\operatorname{Cay}} 

\theoremstyle{plain}
\newtheorem{theorem}{Theorem}[section]
\newtheorem{defin}[theorem]{Definition}
\newtheorem{lemma}[theorem]{Lemma}
\newtheorem{prop}[theorem]{Proposition}
\newtheorem{cor}[theorem]{Corollary}
\newtheorem{question}[theorem]{Question}

\newtheorem*{questionno}{Question}

\theoremstyle{remark}
\newtheorem{note}[theorem]{Remark}
\newtheorem{notes}[theorem]{Remarks}

\title[Cheeger constants and unitarisability of groups]{Asymptotics of Cheeger constants and \\ unitarisability of groups}

\author[M. Gerasimova]{Maria Gerasimova}
\address{Maria Gerasimova, TU Dresden, Germany}
\email{maria.gerasimova@tu-dresden.de}
\author[D. Gruber]{Dominik Gruber}
\address{Dominik Gruber, ETH Zurich, Switzerland}
\email{dominik.gruber@math.ethz.ch}
\author[N. Monod]{Nicolas Monod}
\address{Nicolas Monod, EPFL, Switzerland}
\email{nicolas.monod@epfl.ch}
\author[A. Thom]{Andreas Thom}
\address{Andreas Thom, TU Dresden, Germany}
\email{andreas.thom@tu-dresden.de}

\begin{document}
\begin{abstract}
Given a group $\Gamma$, we establish a connection between the unitarisability of its uniformly bounded representations and the asymptotic behaviour of the isoperimetric constants of Cayley graphs of $\Gamma$ for increasingly large generating sets.

The connection hinges on an analytic invariant $\Lit(\Gamma)\in [0, \infty]$ which we call the \emph{Littlewood exponent}. Finiteness, amenability, unitarisability and the existence of free subgroups are related respectively to the thresholds $0, 1, 2$ and $\infty$ for $\Lit(\Gamma)$. Using graphical small cancellation theory, we prove that there exist groups $\Gamma$ for which $1<\Lit(\Gamma)<\infty$. Further applications, examples and problems are discussed.
\end{abstract}
\date{January 2018}
\maketitle

\tableofcontents

\section{Introduction}
A linear representation $\pi$ of a group $\Gamma$ on a Hilbert space is called \textbf{unitarisable} if $\pi$ is conjugated to a unitary representation by a bounded operator. This implies that $\pi$ is \textbf{uniformly bounded}, that is, $\sup_{g\in \Gamma} \|\pi(g)\|$ is finite. Extending a classical result of Sz.-Nagy~\cite{Sz-Nagy} for $\Gamma=\mathbf Z$, it was shown by several authors~\cites{Day, Dixmier, Nakamura} in 1950 that the converse holds when $\Gamma$ is amenable. That is, amenable groups are \textbf{unitarisable}. It has been open ever since whether this characterises the unitarisability of a group:

\begin{questionno}[Dixmier~\cite{Dixmier}]
Are all unitarisable groups amenable?
\end{questionno}

The first example of a non-unitarisable group was found by Ehrenpreis--Mautner~\cite{Ehrenpreis}, who showed in 1951 that ${\rm SL}_2(\mathbf{R})$ is not unitarisable; it can be deduced that non-abelian free groups are not unitarisable either. In the 1980s, simple and explicit constructions of non-unitarisable representations of free groups were provided, see e.g.~\cites{MZ, PS, MPSZ}. Since unitarisability passes to subgroups, Dixmier's question thus concerns non-amenable groups without free subgroups. The fact that such groups can indeed be non-unitarisable has been confirmed more recently~\cites{EpsteinMonod, Osin, MonodOzawa}.

The starting point of the present article is the connection established by Bo\.zejko--Fendler~\cite{BozejkoFendler} and Wysocza\'nski~\cite{Wysoczanski} between unitarisability, amenability and the space $T_1(\Gamma)$ of \textbf{Littlewood functions}. The latter is the space of all functions $f\colon \Gamma\to\CC$ admitting a decomposition
\begin{equation*}
f(x^{-1}y)=f_1(x,y)+f_2(x,y) \kern10mm \forall x,y \in \Gamma
\end{equation*}
with $f_i\colon \Gamma\times \Gamma\to\CC$ such that both of the following are finite:
$$\sup_x\sum_y|f_1(x,y)| \kern3mm\text{and}\kern3mm \sup_y\sum_x|f_2(x,y)|.$$
The connection is as follows. First, $\Gamma$ is amenable if and only if $T_1(\Gamma) \subseteq \ell^1(\Gamma)$. Secondly, if $\Gamma$ is unitarisable, then $T_1(\Gamma) \subseteq \ell^2(\Gamma)$. Thirdly, if $\Gamma$ contains a non-abelian free subgroup, then $T_1(\Gamma) \nsubseteq \ell^p(\Gamma)$ for all $p<\infty$.

These results prompted us to define the \textbf{Littlewood exponent} $\Lit(\Gamma)  \in [0, \infty]$ of a group $\Gamma$ as follows:
$$\Lit(\Gamma)=\inf\big\{ p : T_1(\Gamma) \subseteq \ell^p(\Gamma) \big\}.$$
It is straightforward that $\Lit(\Gamma)=0$ characterises finite groups and Wysocza\'nski's result~\cite{Wysoczanski} implies that amenable groups satisfy $\Lit(\Gamma)\leq 1$.

Our first result is the converse of the latter statement:

\begin{theorem}\label{beginning}
For every non-amenable group $\Gamma$ there exists $p>1$ such that
$$T_1(\Gamma) \nsubseteq \ell^p(\Gamma).$$
%
\end{theorem}

The situation can therefore be summarised as follows (taking into account a further connection that we shall establish with the rapid decay property of Jolissaint).

\begin{cor}\label{cor:beginning}\leavevmode
\begin{itemize}
\item $\Lit(\Gamma)= 0$ if and only if $\Gamma$ is finite.
\item $\Lit(\Gamma)= 1$ if and only if $\Gamma$ is infinite amenable.
\item $\Lit(\Gamma)\leq 2$ if $\Gamma$ is unitarisable.
\item $\Lit(\Gamma)$ is outside the interval $(1,2)$ if $\Gamma$ has the rapid decay property.
\item $\Lit(\Gamma)= \infty$ if $\Gamma$ contains a non-abelian free subgroup.
\end{itemize}
\end{cor}

A major question is to exhibit groups with $1< \Lit(\Gamma) <\infty$, and particularly with $1< \Lit(\Gamma)  \leq 2$. Concerning the last item of Corollary~\ref{cor:beginning}, we know that it is not a characterisation; adapting~\cites{EpsteinMonod, Osin}, we show:

\begin{theorem}\label{thm:torsion}
There exist finitely generated torsion groups $\Gamma$ with $\Lit(\Gamma)= \infty$.
\end{theorem}

Our next result relates $\Lit(\Gamma)$ to the asymptotics of isoperimetric quantities attached to $\Gamma$ as follows. Given a finite symmetric subset $S\se\Gamma$, consider the (possibly disconnected) Cayley graph $\mathrm{Cay}(\Gamma,S)$. Recall that the \textbf{Cheeger constant} $h(\Gamma,S)$ is defined by
$$h(\Gamma,S)=\inf \limits_F \frac{|\partial_S(F)|}{|F|},$$
where the infimum runs over all non-empty finite subsets $F \se \Gamma$. Define the \textbf{relative maximal average degree} $e(\Gamma,S)$ by
$$e(\Gamma,S)=1-\frac{h(\Gamma,S)}{|S|}.$$
Finally, our asymptotic invariant is
$$\ibet(\Gamma)=-\liminf \limits_{S} \frac{\ln e(\Gamma,S)}{\ln|S|},$$
where the limes inferior is taken over all symmetric finite subsets $S$ of $\Gamma$. By convention, $\ibet(\Gamma)=-\infty$ if $\Gamma$ is finite. Informally, the quantity $\ibet(\Gamma)$ captures the largest exponent such that arbitrarily large sets $S$ can be found with $e(\Gamma,S) \lessapprox |S|^{-\ibet(\Gamma)}$. If $\Gamma$ is free, then one can check that $\ibet(\Gamma)=1$, whereas $\ibet(\Gamma)=0$ if $\Gamma$ is amenable.

\begin{theorem}\label{thm:Lit:ibet}
For any group $\Gamma$ we have $\ibet(\Gamma) = 1-1/\Lit(\Gamma)$.
\end{theorem}

Thus we have a quantitative isoperimetric measure of non-amenability $0< \eta(\Gamma) \leq 1$ for which unitarisability implies $\eta(\Gamma)\leq 1/2$. We have currently no proof that $\eta(\Gamma)$ can take values within $(0,1/2]$; this seems related to the fact that Dixmier's question remains open.

In a similar manner to $\ibet(\Gamma)$, we define
$$r(\Gamma)=-\liminf \limits_{S }\frac{\ln \rho(\Gamma,S)}{\ln|S|},$$
where $\rho(\Gamma,S)\in (0, 1]$ is the spectral radius of the Markov operator associated to $S$. We think of all these invariants as rough guides in the labyrinth of groups that are non-amenable while not containing a non-abelian free subgroup. Combining Theorem~\ref{thm:Lit:ibet} with Cheeger inequalities, we obtain the following.

\begin{cor}\label{chain}
For any infinite group $\Gamma$ we have
$$0 \leq r(\Gamma) \leq \ibet(\Gamma)=1-\frac{1}{\Lit(\Gamma)} \leq 2r(\Gamma) \leq 1.$$
\end{cor}

The following result --- a consequence of graphical small cancellation theory for hyperbolic groups --- shows that the invariant is indeed non-trivial in the sense that there exist groups with $\Lit \not \in \{0,1,\infty\}$.

\begin{theorem} \label{nontrivial}
There exists a group $\Lambda$ with $1 < \Lit(\Lambda)< \infty$.
\end{theorem}

Our construction provides a group $\Lambda$ for which the Cayley graphs $\Cay(\Lambda,S)$ contain images of a Ramanujan graphs of vertex-degree at least $|S|^\epsilon$ for a fixed $\epsilon>0$. The Ramanujan graphs thus provide large finite subsets of vertices containing many internal edges and give upper bounds on $h(\Lambda,S)$. Hence, they can be thought of as analogous to F\o lner sets, providing a certain quantitative degree of amenability in each of the Cayley graphs. This strongly contrasts the way Ramanujan graphs have been utilized thus far in graphical small cancellation constructions: until now, their spectral properties have been used to provide groups satisfying strong negations of amenability, such as non-coarse embeddability into Hilbert spaces \cites{MR1978492,arshdel} and fixed-point properties for actions on $L^p$-spaces \cite{Naor-Silberman}.

Unfortunately, the method employed for the proof cannot be used to establish $\Lit(\Lambda)\leq 2$.

\medskip

Using the connection between the spectral radius and the Littlewood exponent together with Adyan's results~\cite{Adyan}, we can then estimate $\Lit(\Gamma)$ for Burnside groups of large exponent.

\begin{theorem}
Let $B(m,a)$ be the free Burnside group of exponent $a$ on $m$ generators, where $m \geq 2$, $a \geq 665$ and $a$ is odd. Then $\Lit(B(m,a)) \geq 3/2$.
\end{theorem}

The invariants introduced in this article also have applications to estimating the chromatic number of (infinite) Cayley graphs. For convenience, we will say that a graph is $r$-colourable for $r\in \mathbf{R}$ if it is $\left \lfloor r \right \rfloor$-colourable in the usual sense, i.e.\ there exists a colouring of the vertices using ${\lfloor r \rfloor}$ colours such that adjacent vertices get different colours.


\begin{cor}\label{colourcor}
Let $\Gamma$ be a group and $\alpha<\Lit(\Gamma)$. Then there exists arbitrarily large finite symmetric sets $S$ such that ${\rm Cay}(\Gamma,S)$ is $\sqrt[\alpha]{|S|}$-colourable.  
\end{cor}

Another use of $\Lit(\Gamma)$ is related to the first $\ell^2$-Betti number $\beta^{(2)}_1$. Adapting the method of~\cite{EpsteinMonod}, we obtain:

\begin{cor}\label{bettili}
Let $\Gamma$ be a finitely generated group. If $\Gamma$ is residually finite and $\Lit(\Gamma)<\infty$, then $\beta^{(2)}_1(\Gamma)=0$.
\end{cor}



\section{Preliminaries} 
\subsection{Spectral radius and isoperimetry}
We denote by $\nop a p q$ the norm of an operator $a\colon\ell^p(\Gamma)\to \ell^q(\Gamma)$. We extend this notation to any element $a$ of the group ring $\CC[\Gamma]$, considered as a left convolution operator. We further extend the notation to kernels $a\colon \Gamma\times \Gamma\to \CC$ whenever the associated operator given by $v(x)\mapsto \sum_{y\in\Gamma} a(x,y) v(y)$ for $v\in\ell^p(\Gamma)$ and $x\in \Gamma$ is well-defined.

Given a finite symmetric subset $S \subseteq \Gamma$, we denote by $M_S=\frac{1}{|S|}\one_S\in {\mathbf C}[\Gamma]$ the \textbf{Markov operator}. Since $S$ is symmetric, the \textbf{spectral radius} $\rho(\Gamma,S)$ of $M_S$ is realised by $\rho(\Gamma,S) = \nop{M_S}22$. We refer to~\cite{woess} for more information on the spectral radius and its relation to random walks on $\Gamma$.

We consider the Cayley graph ${\rm{\rm Cay}}(\Gamma,S)$, which is connected iff $S$ generates $\Gamma$. Given a finite subset $F\se \Gamma$, we denote by $E_S(F)$ the set of edges of the subgraph of ${\rm{\rm Cay}}(\Gamma,S)$ induced on $F$. The \textbf{edge-boundary} $\partial_S F$ of $F$ is the set of those edges in ${\rm{\rm Cay}}(\Gamma,S)$ that connect $F$ to its complement.  The \textbf{edge-isoperimetric constant} or \textbf{Cheeger constant} of ${\rm{\rm Cay}}(\Gamma,S)$ is defined by
$$h(\Gamma,S) := \inf \limits_{F} \frac{|\partial_S F|}{|F|},$$
where the infimum is taken over all non-empty finite subsets $F \se \Gamma$. We recall the following \emph{Cheeger inequalities}.

\begin{theorem}[Mohar~\cite{Mohar}]\label{thm:Mohar}
$$|S|(1 - \rho(\Gamma,S)) \leq h(\Gamma,S)\leq |S|\sqrt{1-\rho(\Gamma,S)^2}. \eqno{\qed}$$
\end{theorem}

The second inequality above is a special case of Theorem~2.1(a) in~\cite{Mohar}. The first one is easier; a stronger (and more general) statement is Theorem~3.1(a) in~\cite{Mohar}. Alternatively, one can deduce it directly from expanding $\|M_S (\one_F)\|_2 \leq \rho(\Gamma,S) \| \one_F\|_2$; we make a slightly stronger computation below, see Remark~\ref{note:N:e}.

\smallskip
Recall that we defined $e(\Gamma,S)=1-h(\Gamma,S)/|S|$; with this notation, we observe:

\begin{note}\label{rem:Mohar}
The first inequality of Theorem~\ref{thm:Mohar} is equivalent to $e(\Gamma, S)\leq \rho(\Gamma,S)$ whilst the second implies $e(\Gamma, S)\geq \frac12 \rho(\Gamma,S)^2$.
\end{note}

\subsection{Littlewood functions}
We recall the definition introduced by Wysocza\'nski~\cite{Wysoczanski}, extending an idea of Varopoulos~\cite{Varopoulos}; see also~\cite{Pisier2}*{Ch.~2}.

\begin{defin}
The space of \textbf{Littlewood functions} is the space $T_1(\Gamma)$ of all functions $f\colon \Gamma \to \mathbf{C}$ that admit the following decomposition: there exist functions $f_i\colon \Gamma\times \Gamma \to \mathbf{C}$ ($i=1,2$) such that
\begin{equation}\label{eq:T1}
f(x^{-1}y)=f_1(x,y)+f_2(x,y) \kern10mm \forall x,y \in \Gamma
\end{equation}
with $\nop{f_1}\infty\infty<\infty$ and $\nop{f_2}11<\infty$. For such $f$, one defines
\begin{equation}\label{eq:T1-norm}
\|f\|_{T_1(\Gamma)}=\inf \Big\lbrace \nop{f_1}\infty\infty + \nop{f_2}11 \Big\rbrace,
\end{equation}
where the infimum runs over all decompositions~\eqref{eq:T1}.
\end{defin}

For concrete computations with~\eqref{eq:T1-norm}, it is useful to recall that we have
\begin{equation}\label{eq:T1-norm:orig}
\nop{f_1}\infty\infty = \sup_x\sum_y|f_1(x,y)| \kern3mm\text{and}\kern3mm \nop{f_2}11 =\sup_y\sum_x|f_2(x,y)|.
\end{equation}

\begin{notes}\label{rem:T1}
(i). One verifies that $\|f\|_{T_1(\Gamma)}$ is a norm and that moreover $T_1(\Gamma)$ is complete for this norm. (ii). The definition in~\cite{Wysoczanski} considers $\max(\nop{f_1}\infty\infty, \nop{f_2}11)$ instead of the sum in~\eqref{eq:T1-norm}. We caution this defines a quasi-norm only. This is not of much consequence because the two quantities differ at most by a multiplicative constant~$2$. (iii). There is a norm-one embedding $\ell^1(\Gamma) \se T_1(\Gamma)$ because for $f\in\ell^1(\Gamma)$ we can set $f_1(x,y)= f(x^{-1}y)$ and $f_2=0$.
\end{notes}

There is another equivalent norm on the space ${T_1(\Gamma)}$, namely
\begin{equation}\label{eq:N}
N(f) = \sup_{|A|=|B| }\ \frac{1}{|A|}  \sum_{a\in A, b\in B} |f(a^{-1}b)|
\end{equation}
where $A, B\se\Gamma$ range over all finite non-empty subsets of equal size. Adjusting the constants of Remark page~261 in~\cite{Wysoczanski} according to Remark~\ref{rem:T1}(ii) above leads to the following version of Varopoulos' Lemma~5.1 in~\cite{Varopoulos}.

\begin{prop}
For all $f\in T_1(\Gamma)$ we have $\frac12 N(f) \leq \|f\|_{T_1(\Gamma)} \leq 2N(f)$.\qed
\end{prop}

A fundamental property of the norm $N$ is the following.

\begin{prop}\label{prop:N:op}
For every $f\geq 0$ we have $N(f) \leq \nop f22$.
\end{prop}

\begin{proof}
We do not change $\nop f22$ if we consider $f$ as a \emph{right} convolutor on $\ell^2(\Gamma)$. For $A,B$ as in~\eqref{eq:N} we have
$$\sum_{a\in A, b\in B} f(a^{-1}b) = \langle \one_A * f, \one_B \rangle \leq  \|\one_A\|_2 \cdot \|\one_B\|_2 \cdot \nop f22.$$
On the other hand, $\|\one_A\|_2 = \|\one_B\|_2= |A|^{1/2}$; the estimate follows.
\end{proof}

We shall use throughout the following principle.

\begin{lemma}\label{lem:closed:graph}
Let $\|\cdot\|$ one of the equivalent norms on $T_1(\Gamma)$ and let $1\leq p < \infty$.

Then  $T_1(\Gamma) \nsubseteq \ell^p(\Gamma)$ if and only if there is a sequence $(f_n)$ of non-zero elements of $\CC[\Gamma]$ with $\|f_n\|  / \|f_n\|_p \to 0$ as $n\to\infty$. Moreover, we can assume $f_n\geq 0$.
\end{lemma}

\begin{proof}
Suppose that $T_1(\Gamma) \se \ell^p(\Gamma)$ and observe that the inclusion map is continuous in the topology of pointwise convergence. This implies that it has a closed graph also in the product of the norm topologies. Therefore the closed graph theorem ensures that it is norm-continuous and hence no sequence $(f_n)$ as above can exist.

Conversely, suppose $T_1(\Gamma) \nsubseteq \ell^p(\Gamma)$. Since both $N(f)$ and $\|f\|_p$ depend only on $|f|$, there is $f\geq 0$ in $T_1(\Gamma)$ with $\|f\|_p=\infty$. This implies that there is an increasing sequence of finite subsets $F_n\se \Gamma$ with $\|f \one_{F_n}\|_p \to \infty$. However, $N(f \one_{F_n}) \leq N(f)$; therefore, $f \one_{F_n}$ is a sequence with the desired properties.
\end{proof}

\subsection{Limes inferior}
We should clarify our use of liminf in the definition of $\ibet(\Gamma)$ and of $r(\Gamma)$. Since the collection of finite symmetric subsets $S$ of $\Gamma$ is directed under the inclusion order, any real function $f$ on this collection has the usual (possibly infinite) \textbf{limes inferior}
$$\liminf_{S} f(S) := \lim_{S} \inf_{S' \supseteq S} f(S') = \sup_{S} \inf_{S' \supseteq S} f(S').$$
However, in the proofs, we shall use the more ad hoc ``size-wise $\liminf$'' defined by
$$\lim_{n\to\infty } \inf_{|S|\geq n} f(S) = \sup_{n\to\infty } \inf_{|S|\geq n} f(S).$$
The latter is always bounded above by $\liminf_{S} f(S)$ but this inequality can be strict in general. For the particular functions $f$ that we consider in this article, however, we will always have
$$\liminf_{S} f(S) = \lim_{n\to\infty } \inf_{|S|\geq n} f(S).$$
Indeed, observe first that both quantities can be realised respectively as $\lim_{n\to\infty} f(S'_n)$ and $\lim_{n\to\infty} f(S_n)$ for suitable (but unrelated) sequences of finite sets $S'_n, S_n$ with $|S_n|\to\infty$. In order to have the desired equality, it suffices to establish that, given the sequence $(S'_n)$, one can choose the sequence $(S_n)$ with the additional property that for each $n$ there is $m$ with $S_k\supseteq S'_n$ for all $k\geq m$. To this end, it suffices by a diagonal argument to show that given any $\gamma\in\Gamma$, we can replace the sequence $S_n$ by $S_n\cup \{\gamma, \gamma^{-1}\}$.

Since we are in the special case where $f(S)$ is $-\ln e(\Gamma, S)/ \ln|S|$ or $-\ln\rho(\Gamma, S)/ \ln|S|$, this statement follows from straightforward estimates for $e(\Gamma, S)$ and for $\rho(\Gamma, S)$ together with the fact that $|S|$ goes to infinity.

\section{The Littlewood exponent}
\subsection{Finiteness and amenability}
As noted in Remark~\ref{rem:T1}(iii), we have $\ell^1(\Gamma) \se T_1(\Gamma)$ for every group $\Gamma$. One the other hand, if $0<p<1$, then $\ell^p(\Gamma)$ is a proper subspace of $\ell^1(\Gamma)$ unless $\Gamma$ is finite. We can thus record the following.

\begin{lemma}
If $\Gamma$ is infinite, then $\Lit(\Gamma)\geq 1$. If $\Gamma$ is finite, then $\Lit(\Gamma)=0$.\qed
\end{lemma}

On the other hand, Wysocza\'nski has characterised when the inclusion $\ell^1(\Gamma) \se T_1(\Gamma)$ is proper.

\begin{theorem}[Wysocza\'nski~\cite{Wysoczanski}*{Thm.~1}]\label{wys}
The group $\Gamma$ is amenable if and only if $\ell^1(\Gamma)=T_1(\Gamma)$.\qed
\end{theorem}

In particular, $\Lit(\Gamma)\leq 1$ for amenable groups. We shall establish the converse; this requires quantitative estimates for the $T_1$-norm which we shall prove using the following result of one of the authors~\cite{Thom:rad}.

\begin{theorem}[Cor.~6 in~\cite{Thom:rad}]\label{spectralradius}
For every non-amenable group $\Gamma$ there is $\epsilon>0$ such that there exist arbitrarily large finite symmetric subsets $S$ with $\rho(\Gamma, S)<|S|^{-\epsilon}$.\qed
\end{theorem}

We can now prove the characterisation of amenability stated in Theorem~\ref{beginning}.

\begin{proof}[Proof of Theorem~\ref{beginning}]
Let $\epsilon>0$ be as provided by Theorem~\ref{spectralradius}. Then any $p$ with $1<p<1/(1-\epsilon)$ will do. Suppose indeed for a contradiction that we have $T_1(\Gamma) \se \ell ^p(\Gamma)$. In view of Proposition~\ref{prop:N:op} and Lemma~\ref{lem:closed:graph}, there is a constant $c$ such that
$$|S|^{\frac1p} = \|\one_S\|_p \leq c N(\one_S) \leq c \nop{\one_S}22 = c |S| \rho(\Gamma, S)$$
holds for every finite symmetric subset $S\se\Gamma$. Applying Theorem~\ref{spectralradius}, we obtain that $|S|^{\frac1p}$ is bounded by $c |S|^{1-\epsilon}$ for arbitrarily large sets $S$, which implies $1/p\leq 1-\epsilon$, a contradiction.
\end{proof}

\subsection{Subgroups, quotients and free groups}
In view of Theorem~\ref{beginning}, a value $\Lit(\Gamma)>1$ is one of the ways to measure non-amenability quantitatively. Therefore, the stability properties of $\Lit(\Gamma)$ are of interest.

\begin{prop}\label{aforsubgroups}
If $\Lambda$ is a subgroup of $\Gamma$, then $\Lit(\Lambda) \leq \Lit(\Gamma)$.
\end{prop}

\begin{proof}
Since for all $p$ the space $\ell^p(\Lambda)$ coincides with the subspace of elements of $\ell^p(\Gamma)$ supported on $\Lambda$, it suffices to justify that every $f\in T_1(\Lambda)$ yields an element of $T_1(\Gamma)$ after extending it by zero outside $\Lambda$. This is straightforward if we use the norm $N$. If we use the $T_1$-norm, we just need to extend the two functions $f_i(x,y)$ by restricting to the case where $x^{-1}$ and $y$ belong to the same coset $\Lambda$ in $\Gamma$ (as done in the proof of Lemma~2.7 of~\cite{Pisier2}).
\end{proof}

\begin{prop}\label{aforquotients}
If $\Lambda$ is a quotient group of $\Gamma$, then $\Lit(\Lambda) \leq \Lit(\Gamma)$.
\end{prop}

\begin{proof}
Given $f\in T_1(\Lambda)$ and $p>\Lit(\Gamma)$, we shall produce an element $\widetilde f\in T_1(\Gamma)$ with $\|\widetilde f\|_p= \| f\|_p$; this implies the statement. Let $R\se \Gamma$ be a set of representatives for the quotient map $\pi\colon\Gamma\to\Lambda$. We define $\widetilde f  (x) = f(\pi(x))$ if $x\in R$ and $\widetilde f  (x) =0$ otherwise. The relation $\|\widetilde f\|_p= \| f\|_p$ holds and we need to prove that $\widetilde f$ belongs to $T_1(\Gamma)$.

Define thus $\widetilde f_i(x,y)= f_i(\pi(x),\pi(y))$ if $x^{-1}y \in R$ and~$0$ otherwise, for $i=1,2$. Then $\widetilde f_i$ provides a decomposition as in~\eqref{eq:T1} for $\widetilde f$. We now check the finiteness of $\nop{\widetilde f_1}\infty\infty$ using the formula~\eqref{eq:T1-norm:orig}:
$$\sup_{x\in\Gamma}\sum_{y\in\Gamma}|\widetilde f_1(x,y)| = \sup_{x\in\Gamma}\sum_{y\in x R}|f_1(\pi (x),\pi(y))|.$$
The right hand side is exactly $\nop{f_1}\infty\infty$ because each $xR$ is itself a set of representatives for $\pi$. The computation for $\nop{\widetilde f_2}11$ is similar but uses the fact that  $yR^{-1}$ is a set of representatives.
\end{proof}

\begin{prop}
If $\Lambda$ is a finite index subgroup of $\Gamma$, then $\Lit(\Lambda) = \Lit(\Gamma)$.
\end{prop}

\begin{proof}
In view of Proposition~\ref{aforsubgroups}, what we need to prove is that for any $p>\Lit(\Lambda)$ and any $f\in T_1(\Gamma)$, the norm $\|f\|_p$ is finite. Let $R\se \Gamma$ be a set of representatives of $\Lambda$-cosets and define $f^r\colon\Lambda\to\CC$ by $f^r(x) = f(xr)$ for each $r\in R$. If $f_1, f_2$ provide a decomposition as in~\eqref{eq:T1}, we define $f^r_i\colon \Lambda\times\Lambda\to\CC$ by $f^r_i(x,y)=f_i(x, y r)$. This witnesses that $\| f^r\|_{T_1(\Lambda)} \leq \| f\|_{T_1(\Gamma)} $ and hence each $f^r$ is in $\ell^p(\Lambda)$. Viewing $f$ as a sum of translates of the various $f^r$ to the corresponding cosets $\Lambda r\se\Gamma$, we conclude that $f\in\ell^p(\Gamma)$ since $R$ is finite.
\end{proof}

\begin{theorem}\label{thm:Lit:free}
If $\Gamma$ contains a non-abelian free subgroup, then $\Lit(\Gamma)=\infty$.
\end{theorem}

\begin{proof}
Any non-abelian free group contains a free group $F_\infty$ of countable rank. Therefore, by Proposition~\ref{aforsubgroups}, it suffices to recall that $\Lit(F_\infty)=\infty$, which goes back to Wysocza\'nski~\cite{Wysoczanski}. For convenience, we recall the argument. Let $T$ be a basis of $F_\infty$ and set $S=T\cup T^{-1}$. Then $\one_S$ does not belong to any $\ell^p$ with $p<\infty$. We decompose
$$\one_S(x^{-1}y)  = f_1(x,y) + f_2(x,y)$$
as follows. Set $f_1(x,y)=1$ if $x^{-1}y\in S$ with $y$ shorter than $x$ in the $S$-word-length (and~$0$ otherwise). Then $x$ determines $y$ and hence $\sup_x\sum_y|f_1(x,y)| =1$. We have a similar bound for $f_2(x,y)=f_1(y,x)$ and thus $\| \one_S \|_{T_1(\Gamma))}\leq 2$.
\end{proof}

\section{Asymptotic isoperimetry and the Littlewood exponent}
\subsection{Another norm on \texorpdfstring{$T_1$}{T1}}
We modify the norm $N$ on ${T_1(\Gamma)}$ as follows:
\begin{equation}\label{eq:N'}
N'(f) = \sup_{F} \frac{1}{|F|}  \sum_{a,b\in F} |f(a^{-1}b)|
\end{equation}
where $F\se\Gamma$ ranges over all non-empty finite subsets.

\begin{lemma}\label{lem:N'}
The norm $N'$ is equivalent to $N$ and hence to $\|\cdot\|_{T_1(\Gamma)}$; specifically:
$$\frac12 N(f) \leq N'(f)\leq N(f).$$
\end{lemma}

\begin{proof}
The first equality follow by setting $F:=A\cup B$ for $A,B$ as in the definition of $N(f)$. For the second, set $A=B:=F$.
\end{proof}

For characteristic functions, the norm $N'$ has the following geometric interpretation.

\begin{prop}\label{prop:N':e}
We have $N'(\one_S)=e(\Gamma, S) |S|$ for every finite symmetric $S\se \Gamma$.
\end{prop}

\begin{note}\label{note:N:e} 
If we combine this identity with the second inequality of Lemma~\ref{lem:N'} and the inequality $N(\one_S)\leq \nop{\one_S}22$ of Proposition~\ref{prop:N:op}, we conclude that $e(\Gamma, S) |S| \leq \nop{\one_S}22 =\rho(\Gamma, S) |S|$. This is equivalent to the first Cheeger inequality in Theorem~\ref{thm:Mohar}.
\end{note}

For the proof, we record the following identity, which results from counting all edges adjacent to any element of $F$.

\begin{lemma}\label{lem:count:E}
For any finite subsets $F$ and $S=S^{-1}$ of the group $\Gamma$, we have
\begin{equation*}
|F| |S|  =  |\partial_S F| + 2 |E_S(F)| - |L_S(F)|,
\end{equation*}
where $L_S(F)$ denotes the set of loops in $E_S(F)$. In particular, it follows that
\begin{equation*}
e(\Gamma,S)= \sup_F \frac{2 |E_S(F)| - |L_S(F)|}{|F||S|}.
\end{equation*}
\end{lemma}

The second statement of the lemma explains why we called $e(\Gamma,S)$ the \emph{relative maximal average degree} (loops are given half-weight for vertex-degree).

\begin{proof}[Proof of Proposition~\ref{prop:N':e}]
The sum $\sum_{a,b\in F} \one_S(a^{-1} b)$ counts the elements of $E_S(F)$ twice, except that loops are counted once. Therefore, its value is $|F| |S| -  |\partial_S F|$ by Lemma~\ref{lem:count:E}. It follows that $N'(\one_S)= |S| - h(\Gamma, S)$, as desired.
\end{proof}

\begin{note}
 Lemma~\ref{lem:count:E} together with the first Cheeger inequality in Theorem~\ref{thm:Mohar} also gives
\begin{equation*}
|E_S(F)| - \frac12 |L_S(F)|  \leq \frac12 |F| |S| \rho(\Gamma,S).
\end{equation*}
Note that the availability of this kind of bound crucially depends on the fact that we are working on a regular graph.
\end{note}

\subsection{Proof of Theorem~\ref{thm:Lit:ibet}}
We will say that a function is a \textbf{box-function} if is is a multiple of a characteristic function. We need the following ``box trick''.

\begin{lemma}\label{lemma:box}
Let $0<q<p$. For every $f\geq 0$ in $\ell^p(\Gamma)$ there is a box function $f^\square$ with $0\leq f^\square \leq f$ and satisfying
$$\|f^\square\|_q \geq \zeta(p/q)^{-\frac1p} \, \|f\|_p.$$
\end{lemma}

\noindent
(The constant is optimal: consider $n\mapsto n^{-1/q}$ in the proof below.)

\begin{proof}
Since $f$ has countable support and since the statement is invariant under any permutation of the support, it suffices to give a proof for the case of a non-increasing function $f\colon\NN^*\to\RR_+$. For such $f$, there is $n$ such that
\begin{equation}\label{eq:box}
f(n) \geq \zeta(p/q)^{-\frac1p} \, \|f\|_p \  n ^{-\frac1q}
\end{equation}
because otherwise we obtain a contradiction by summing over all $n$ the $p$-powers of both sides in~\eqref{eq:box}. We now define $f^\square(m)=f(n)$ for $m\leq n$ and $f^\square(m)=0$ otherwise. Then $\|f^\square\|_q = f(n) n^{1/q}$ and hence $\|f^\square\|_q$ satisfies the statement of the lemma thanks to~\eqref{eq:box}.
\end{proof}

\begin{proof}[Proof of Theorem~\ref{thm:Lit:ibet}]
Fix any $0<q<p<\Lit(\Gamma)$. By Lemma~\ref{lem:closed:graph}, there is a sequence $f_n\geq 0$ in $\CC[\Gamma]$ with $N'(f_n) / \|f\|_p \to 0$. Thus Lemma~\ref{lemma:box} provides us with a sequence of finite sets $S_n$ such that $N'(\one_{S_n}) / \|\one_{S_n}\|_q$ tends to zero; we used here that $0\leq f^\square \leq f$ implies $N'( f^\square) \leq N'(f)$. We can assume $S_n$ symmetric by replacing it with $S_n \cup S_n^{-1}$ since this introduces at most a factor~$2$ in the norm. We have in particular
$$\lim_{m\to\infty} \inf_{|S|\geq m} \frac{\ln N'(\one_{S}) }{\ln \|\one_{S}\|_q } \leq \limsup_{n\to\infty} \frac{\ln N'(\one_{S_n}) }{\ln \|\one_{S_n}\|_q } \leq 1.$$
By Proposition~\ref{prop:N':e}, the numerator of the left fraction is $\ln e(\Gamma, S) + \ln |S|$. The denominator is $\frac 1q \ln|S|$. Therefore
$$\frac{\ln e(\Gamma, S)}{\ln |S|} \leq\frac1q -1$$
for sets $S$ of arbitrarily large size. Since $q$ can be taken arbitrarily close to $\Lit(\Gamma)$, we conclude $\ibet(\Gamma) \leq 1-1/\Lit(\Gamma)$.

\smallskip
Suppose for a contradiction that the inequality is strict. We can then choose $p>q>\Lit(\Gamma)$ with $\ibet(\Gamma)> 1-1/p$. By definition of $\ibet(\Gamma)$, there is a sequence of finite symmetric sets $S_n$ in $\Gamma$ with $|S_n|\to\infty$ and $\ln e(\Gamma, S_n) / \ln|S_n|$ bounded by $1-1/p$ for all $n$. Using again Proposition~\ref{prop:N':e}, this bound is equivalent to $N'(\one_{S_n})\leq  |S_n|^{1/p}$. Since $|S_n|\to\infty$ and $p>q$ it follows that $N'(\one_{S_n}) /  |S_n|^{1/q} \to 0$. On the other hand, $|S_n|^{1/q}  = \|\one_{S_n}\|_q$ is bounded by a constant times $N'(\one_{S_n})$ since $q>\Lit(\Gamma)$; this is a contradiction.
\end{proof}

\section{An example of a group \texorpdfstring{$\Lambda$}{} with \texorpdfstring{$1< \Lit(\Lambda)<\infty$}{1<Lit<infty}} \label{smallcanc}

The aim of this section is to prove Theorem~\ref{nontrivial}  as an application of graphical small cancellation theory for hyperbolic groups --- by now a common source of exotic groups. Indeed, we construct a monster group with the property that we can control the isoperimetric behaviour for \emph{every} symmetric subset. The key step in the inductive construction is the following theorem.

\begin{theorem}\label{theorem:hyperbolic_quotient}
There exists $\epsilon>0$ with the following property. Suppose $\Gamma$ is a non-elementary torsion-free hyperbolic group, $S,K \se \Gamma$ finite symmetric subsets. Then, there exists a non-elementary torsion-free hyperbolic quotient $\pi \colon \Gamma\to\Lambda$ such that $\pi$ is injective on $K$ and $e(\Lambda,\pi(S)) \geq \frac12 |\pi(S)|^{-1+\epsilon}$.
\end{theorem}

It is now immediate to prove the main result of this section.

\begin{proof}[Proof of Theorem~\ref{nontrivial}]
Consider a non-elementary torsion-free hyperbolic Kazhdan group $\Gamma=\Gamma_0$ and enumerate all finite symmetric subsets of $\Gamma$ in form of a sequence $(\Sigma_n)_{n \geq 1}$. Set $K_0= \varnothing.$ For each natural number $n \geq 1$, we construct a quotient $\pi_n \colon \Gamma_{n-1} \to \Gamma_n$ and a finite subset $K_n \se \Gamma_n$ as follows: Consider the image $S_n$ of $\Sigma_n$ in $\Gamma_{n-1}$ and take $K_{n-1} \se \Gamma_{n-1}$ as constructed by induction. By the preceding theorem, there exists a torsion-free hyperbolic quotient $\pi_n \colon \Gamma_{n-1} \to \Gamma_n$, such $e(\Gamma_n,\pi_n(S_n)) \geq \frac12 |\pi_n(S_n)|^{-1 + \epsilon}$ witnessed, using Lemma~\ref{lem:count:E}, by a finite set $F_n \se \Gamma_n$ with
$$\frac{|E_{\pi_n(S_n)}(F_n)|}{|F_n|} \geq \frac15|\pi_n(S_n)|^{\epsilon}$$
and $\pi_n$ is injective on $K_{n-1}$. We set $K_n:=\pi_n(K_{n-1} \cup S_n) \cup F_n \cup F_n^{-1}$. Consider now the inductive limit
$\Lambda := \lim_{n \to \infty} \Gamma_n$ along the maps $\pi_n \colon \Gamma_{n-1} \to \Gamma_n$.
For every finite symmetric subset $S \se \Lambda$, it easily follows, again using Lemma~\ref{lem:count:E}, that we have $e(\Lambda,S) \geq \frac15|S|^{-1+\epsilon}$ and thus $\eta(\Lambda) \leq 1-\epsilon<1$ or equivalently $\Lit(\Lambda)\leq1/\epsilon< \infty$.

Now, $\Lambda$ is an infinite Kazhdan group and hence non-amenable. We conclude from Theorem~\ref{beginning} that $\Lit(\Lambda)>1$.
\end{proof}

We now proceed with the proof of Theorem~\ref{theorem:hyperbolic_quotient}. We invoke the following two propositions from small cancellation theory:

\begin{prop}\label{prop:tarski_monster} Let $\Gamma$ be a non-elementary torsion-free hyperbolic group and $S$ a symmetric finite subset such that the subgroup generated by $S$ is non-elementary. Let $K$ be a finite subset of $\Gamma$. Then there exists a non-elementary torsion-free hyperbolic quotient $\pi \colon \Gamma\to\Lambda$ such that $\pi(S)$ generates $\Lambda$, and $\pi$ is injective on $K$.
\end{prop}

\begin{proof}
This is a direct consequence of \cite{Olshanskii}*{Theorem~1}.
\end{proof}

The next proposition is a consequence of the inductive step of Gromov's construction of random groups that contain (in a certain sense) expander graphs~\cite{MR1978492}. To make the dependencies of the involved constants clear, we shall follow the detailed account of Gromov's result given by Arzhantseva--Delzant~\cite{arshdel}. We will use Coulon's explanation \cite{Coulon} of the small cancellation theorem involved in the construction.

We begin by explaining the setup. Given a graph $\Theta$ whose edges are oriented and a finite symmetric subset $S$ of a group $\Gamma$, a labelling of $\Theta$ by $S$ is a map $\ell\colon E(\Theta)\to S$. We shall identify two $S$-labelled graphs $\Theta$ and $\Theta'$ if $\Theta'$ can be obtained from $\Theta$ by a collection of moves of the form: flip the orientation of an edge and replace its label $s\in S$ by $s^{-1}\in S$. Notice that the Cayley graph $\Cay(\Gamma,S)$ carries a natural labelling by $S$. 

Given an $S$-labelled graph $\Theta$ and an $S'$-labelled graph $\Theta'$ together with a map $S\to S'$, there is an obvious notion of label-preserving graph homomorphism $\Theta\to\Theta'$. When we say label-preserving graph isomorphism, we shall also require that $S\to S'$ is a bijection. If $p$ is a path in $\Theta$ then we can write $p=(e_1^{\epsilon_1},e_2^{\epsilon_2},\dots,e_k^{\epsilon_k})$, where each $e_i$ is an oriented edge and $\epsilon_i\in\{\pm1\}$. The label of $p$ is defined as $$\ell(p)=\ell(e_1)^{\epsilon_1}\ell(e_2)^{\epsilon_2}\dots\ell(e_k)^{\epsilon_k}.$$ Denote by $\Gamma/\Theta$ the quotient of $\Gamma$ by the normal closure of the image in $\Gamma$ of all labels of closed paths in $\Theta$. Then, for each connected component of $\Theta$, the labelling induces a label-preserving graph homomorphism to $\Cay(\Gamma/\Theta,\pi(S))$.

The uniform random labelling of $\Theta$ by $S$ is the probability distribution on the set of labellings of $\Theta$ obtained as the product distribution from the uniform distribution on $S$ for each edge. In other words, given an edge $e$, for each $s\in S$ we label $e$ by $s$ with probability $1/ |S|$, and labels of distinct edges are independent. As $S$ is symmetric (and considering the identification discussed above), this distribution does not depend on the orientation of $\Theta$. Thus, if $\Theta$ was not a priori not oriented, we can simply endow it with any fixed orientation. 

The girth of a graph is the length of a shortest homotopically non-trivial closed path if such a path exists and $\infty$ otherwise. We denote the diameter of a space $\Theta$ by $\diam(\Theta)$. 

\begin{prop}\label{prop:small_cancellation}
Let $\delta>0$ and $A>0$. Then there exist $\nu>0$ and $\epsilon>0$ with the following property. Suppose $\Gamma$ is a non-elementary torsion-free hyperbolic group and $S$ is a finite symmetric generating subset such that $\rho(\Gamma,S)\leq|S|^{-\delta}$. Let $d\leq|S|^{2\epsilon}$ and let $(\Theta_n)_{n\in\NN}$ be a sequence of finite connected graphs of vertex degree at most $d$ such that, for all $n$, we have
$$\diam(\Theta_n)\leq A\girth(\Theta_n)$$
and such that $|V(\Theta_n)|\to\infty$.
Then, with probability tending to $1$ as $n \to \infty$, for the uniform random edge-labelling of $\Theta_n$ by $S$, the following hold.
\begin{itemize}
 \item[(i)] The group $\Gamma/\Theta_n$ is non-elementary torsion-free hyperbolic.
 \item[(ii)] The map $\pi \colon \Gamma\to\Gamma/\Theta_n$ is injective on a ball of radius $\nu\girth(\Theta_n)$ w.r.t. $S$.
 \item[(iii)] The map $\pi_1(\Theta_n)\to\Gamma$ induced by the labelling is injective.
 \item[(iv)] Let $T_n$ be an image of the universal cover $\tilde\Theta_n$ of $\Theta_n$ in $\Cay(\Gamma,S)$ 
 and $H_n$ a corresponding conjugate of
 $\pi_1(\Theta_n)$ in $\Gamma$. Then, for any label-preserving graph homomorphism $f \colon \Theta_n\to\Cay(\Gamma/\Theta_n,\pi(S))$, we have $f(\Theta_n)\cong H_n\backslash T_n$ as labelled graphs.
\end{itemize}
\end{prop}

The proof following \cite{arshdel} consists of two ingredients: the first ingredient is the study of geometry of the image of the random words read on $\Theta_n$ in $\Gamma$ \cite{arshdel}*{Section~5}. Notice here that the symmetric measure $\mu$ on $\Gamma$ considered in \cite{arshdel} does not exactly correspond to our definition in the case that $e\in S$. However, the computations in \cite{arshdel} also apply (replacing their $2k$ by our $|S|$) since, in our situation, the measures on $S$ we use to define the Markov operator (which gives the spectral radius) and to define the random labelling coincide with each other. The only adjustment from \cite{arshdel} is that Kesten's lower bound on $\rho(\Gamma,S)$ used below looks slightly different in our case. 

The second ingredient is an application of results of geometric small cancellation theory. For this, we chose to use Coulon's version \cite{Coulon}*{Theorem~7.10} of \cite{arshdel}*{Theorem~3.10}. See \cite{Coulon}*{page~325} for a remark on the compatibility of the approaches in \cite{arshdel} and in \cite{Coulon}.

\begin{proof}[Proof of Proposition~\ref{prop:small_cancellation}] Let $\rho_0$ the constant of \cite{Coulon}*{Theorem~7.10}, and let $\delta_2$ and $\Delta_2$ be the values obtained from that theorem for  $k=\max\{8/\delta,1\}$ and $\rho=\rho_0$. Thus, $\delta_2$ and $\Delta_2$ only depend on $\delta$. We denote $\lambda:=\max\{8/\delta,1\}$ and by $\delta_\Gamma$ the hyperbolicity constant of $\Cay(\Gamma,S)$. By \cite{Coornaert}*{Chapter~3}, see also \cite{arshdel}*{Theorem~3.7} and \cite{Coulon}*{Proposition~7.9}, there exist constants $C_{loc}(\lambda),C_{QC}(\lambda)>0$ depending only on $\lambda$ and constants $D_{loc}(\lambda,\delta_\Gamma),D_{QC}(\lambda,\delta_\Gamma)>0$ depending only on $\lambda$ and $\delta_\Gamma$ such that any $(D_{loc}(\lambda,\delta_\Gamma)+C_{loc}(\lambda)\eta)$-local $(\lambda,\eta)$-quasi-geodesic in a $\delta_\Gamma$-hyperbolic geodesic space is a global $(2\lambda,\eta)$-quasi-geodesic whose image is $(D_{QC}(\lambda,\delta_\Gamma)+C_{QC}(\lambda)\eta)$-quasi-convex.  See \cite{Coulon}*{Definition~7.8} for the definition of local quasi-geodesic we use. Set $\xi_0:=\min\{\frac{1}{4C_{loc}(\lambda)},\frac{1}{4\lambda},\frac{\delta_2}{4\lambda C_{QC}(\lambda)},\frac{\Delta_2}{1000\lambda}\}$ and $A':=\max\{1,A\}$. We prove the claim of the proposition for
$$\nu:=\frac{\rho}{80\pi\lambda\sinh(\rho)}\kern3mm \text{and} \kern3mm\epsilon:=\frac{\delta\xi_0}{4(A'+\xi_0)}.$$
Suppose $d\leq|S|^{2\epsilon}$. Let $b:=-\frac{3}{4}\ln(\rho(\Gamma,S))$. We claim: $(\Theta_n)_{n\in\NN}$ is $(b,\xi_0)$-thin in the sense of~\cite{arshdel}*{Definition~5.3}. (Notice that the definition of ``$b$-thin'' in~\cite{arshdel} carries an implied constant $\xi_0$.) For this, it is sufficient to verify, denoting $\rho_n:=\girth(\Theta_n)$: for each $n\in\NN$ and each $\xi\in[\xi_0,1/2)$, the number of simple paths in $\Theta_n$ of length $\xi\rho_n$, denoted $b_n(\xi\rho_n)$, satisfies $\frac{1}{d}b_n(\xi\rho_n)\leq \exp(b\xi_0\rho_n)$. Observe that $(A'+\xi)/(A'+\xi_0)< 3/2$. As desired:
$$\frac{1}{d}b_n(\xi\rho_n)\leq d^{(A'+\xi)\rho_n}\leq |S|^{\frac{A'+\xi}{A'+\xi_0}\frac{\delta\xi_0\rho_n}{2}}< |S|^{\frac{3}{4}\delta\xi_0\rho_n}\leq\rho(\Gamma,S)^{-\frac{3}{4}\xi_0\rho_n}=\exp(b\xi_0\rho_n).$$
As shown in~\cite{arshdel}*{Lemma~5.7}, with probability going to 1 as $n\to\infty$, the uniform random labelling of $\Theta_n$ by $S$ satisfies that the map $\tilde\Theta_n\to\Cay(\Gamma,S)$ is a $(\rho_n/2)$-local $(\lambda_0,(2/\lambda_0)\xi_0\rho_n)$-quasi-isometric embedding, where 
$$\lambda_0:=-\frac{2\ln(|S|-1)}{b+\ln(\rho(\Gamma,S))}=-\frac{8\ln(|S|-1)}{\ln(\rho(\Gamma,S))}\leq\frac{8\ln(|S|-1)}{\delta\ln|S|}<\frac{8}{\delta}\leq\lambda.$$
Here we use that $|S|\geq 4$ since $\Gamma$ is non-elementary and has no element of order two. We also have $\lambda_0\geq2$ by the standard bound $(|S|-1)^{1/2}/|S|\leq\rho(\Gamma,S)$ of Kesten~\cite{Kesten} for any finite symmetric subset of a group. Thus, the map is a $(\rho_n/2)$-local $(\lambda,\xi_0\rho_n)$-quasi-isometric embedding.

By our choice of $\xi_0$, we have $C_{loc}(\lambda)\xi_0\rho_n\leq\rho_n/4$. Hence, by the aforementioned result of \cite{Coornaert}*{Chapter~3}, if $\rho_n/4\geq D_{loc}(\lambda,\delta_\Gamma)$ (i.e. if $n$ and hence $\rho_n$ is large enough), the map $\tilde\Theta_n\to\Cay(\Gamma,S)$ is a $(2\lambda,\xi_0\rho_n)$-quasi-isometric embedding and thus, by our choice of $\xi_0$, it is a $(2\lambda,\rho_n/(4\lambda))$-quasi-isometric embedding. Therefore,  the shortest length in $\Gamma$ of an element represented by the label of a homotopically non-trivial closed path in $\Theta_n$ is at least $\rho_n/(4\lambda)>0$, showing that (iii) holds. In the notation of \cite{Coulon}*{Theorem~7.10} we have $T(\mathcal Q)\geq \rho_n/(4\lambda)$. In particular, if $\rho_n$ is large enough, then $\delta_\Gamma/T(\mathcal Q)\leq \delta_2$.

We also deduce that the image of $\tilde\Theta_n$ is $(D_{QC}(\lambda,\delta_\Gamma)+C_{QC}(\lambda)\xi_0\rho_n)$-quasi-convex, and we have $C_{QC}(\lambda)\xi_0\rho_n\leq\delta_2\rho_n/(4\lambda)$. Thus, if $\rho_n$ is large enough, we have (denoting the quasi-convexity constant by $\alpha$ as in \cite{Coulon}*{Theorem~7.10}) $\alpha/T(\mathcal Q)\leq2\delta_2< 10\delta_2$.

Furthermore, as shown in the proof of~\cite{arshdel}*{Lemma~5.8} (using $\xi_0<(\Delta_2/(4\lambda))/200$), with probability going to 1 as $n\to\infty$, in the notation of \cite{Coulon}*{Theorem~7.10}, we have $\Delta'(\mathcal Q)\leq(\Delta_2/(4\lambda))\rho_n$ and thus $\Delta'(\mathcal Q)/T(\mathcal Q)\leq\Delta_2$. (Note that the metric estimates on [\cite{arshdel}, page~21] indeed provide an upper bound for $\Delta'(\mathcal Q)$ as defined on \cite{Coulon}*{page~324}.) Hence, with probability going to 1 as $n\to\infty$, all assumptions of~\cite{Coulon}*{Theorem~7.10} are fulfilled, and we deduce that (i) and (ii) are also satisfied.

Finally, (iv) is a consequence of the proof of \cite{Coulon}*{Theorem~6.11}, which applies here as explained in the proof of \cite{Coulon}*{Theorem~7.10}. Let $T_n$ be the image of $\tilde\Theta_n$ in $\Cay(\Gamma,S)$ obtained by sending an element of the fiber of a base vertex $v$ in $\Theta_n$ to $e\in\Gamma$ and $H_n$ the image of $\pi_1(\Theta_n,v)$ in $\Gamma$ defined by the labelling. Then $H_n$ acts on $T_n$ by left-multiplication. In the proof of \cite{Coulon}*{Theorem~6.11}, Coulon constructs a space $\dot X$ by rescaling $\Cay(\Gamma,S)$ and attaching topological cones of radius $\rho$ (endowed with a certain hyperbolic metric) to each $\Gamma$-translate of an appropriate neighbourhood $Z_n$ of $T_n$. If $\mathcal R$ denotes the set $\{(gH_ng^{-1},gc):g\in\Gamma\}$, where $c$ denotes the apex of the cone over $Z_n$, then $\mathcal R$ is a $2\rho$-rotation family for the isometric action of $\Gamma$ on $\dot X$. 

Consider $K_n$ the normal closure of $H_n$ in $\Gamma$ and suppose, for some $g\in K_n$ and $x,y$ vertices of $T_n$, we have $gx=y$. For $0<r<\rho$, consider the points $x'$ and $y'$ above $x$ and $y$, respectively, in the cone over $Z_n\supseteq T_n$ at distance $r$ from $c$. Then, since $gx'=y'$, we have $d(c,gc)\leq d(c,y')+d(gx',gc)=2r<2\rho$. This implies that $c=gc$ because, as $\mathcal R$ is a $2\rho$-rotation family, the translates of $c$ are $2\rho$-separated. Hence, $g$ is in the stabilizer of $c$ in $K_n$, which is $H_n$ by \cite{Coulon}*{Corollary~3.13}. Thus, on the level of vertex sets, the quotient of $T_n$ given by the map $\pi:\Gamma\to K_n\backslash \Gamma$ is indeed $H_n\backslash T_n$. (Recall that $K_n\backslash\Gamma=\Gamma/\Theta_n$.) On the level of labelled graphs, our claim follows with the additional observation that if the injectivity radius obtained in (ii) is greater than 2, then $\pi$ restricted to $S$ is a bijection onto $\pi(S)$. This holds if $\rho_n$ is large enough.
\end{proof}

\begin{proof}[Proof of Theorem~\ref{theorem:hyperbolic_quotient}]We derive our Theorem~\ref{theorem:hyperbolic_quotient} from the two propositions. Let $0<\delta<1/2$ be arbitrary. By~\cite{lubotzky}*{Theorem~7.3.12}, there exists a universal constant $A>0$ such that for every odd prime $p$, there exists a sequence of $(p+1)$-regular graphs $(\Theta_n)_{n\in \NN}$ satisfying the conditions of Proposition~\ref{prop:small_cancellation}. We show that if $\epsilon$ is obtained from Proposition~\ref{prop:small_cancellation} for these values of $\delta$ and $A$, then $\epsilon'=\min\{1-2 \delta,\epsilon\}$ satisfies the claim of Theorem~\ref{theorem:hyperbolic_quotient}.

\vspace{0.1cm}

Let $\Gamma$ be a torsion-free non-elementary hyperbolic group, $K$ a finite subset and $S$ a finite symmetric subset. If $S$ generates an elementary subgroup, then this subgroup is in particular amenable, and we have $\rho(\Gamma,S)=1$, i.e.\ we have $e(\Gamma,S)=1$ and there is nothing to prove taking $\Lambda=\Gamma$. Thus, assume $S$ generates a non-elementary subgroup. If $S$ does not generate $\Gamma$, then we apply Proposition~\ref{prop:tarski_monster} to obtain a quotient $\pi_0\colon\Gamma\to \Lambda_0$ that is injective on $K$, such that $\Lambda_0$ is non-elementary torsion-free hyperbolic and $\pi_0(S)$ generates $\Lambda_0$. It is then sufficient to prove the claim of the theorem for $\Lambda_0$, $\pi_0(S)$, and $\pi_0(K)$. Thus, we henceforth assume that $\Gamma$ is generated by $S$.

If $\rho(\Gamma,S)\geq|S|^{-\delta}$, then $e(\Gamma,S) \geq \frac12 \rho(\Gamma,S)^2 \geq\frac12 |S|^{-2\delta} \geq \frac12|S|^{-1 + \epsilon'}$ and we are done taking $\Lambda=\Gamma$. Observe from Lemma~\ref{lem:count:E} that, whenever $F\subset\Lambda$ is finite, then $e(\Gamma,S)\geq\frac{|E_S(F)|}{|F||S|}$. If $|S|^\epsilon\leq 2$, then, since any Cayley graph contains either a cycle or an infinite line,
$e(\Gamma,S) \geq 1/|S| \geq \frac{1}{2}|S|^{-1+\epsilon} \geq \frac12|S|^{-1+\epsilon'}$.
Thus, we henceforth assume $\rho(\Gamma,S)<|S|^{-\delta}$ and $|S|^\epsilon>2$.

Since $|S|^\epsilon>2$, using Bertrand's Postulate~\cite{tcheb}*{p.~382}, we find an odd prime $p$ with $\lfloor |S|^\epsilon\rfloor<p<2\lfloor|S|^\epsilon\rfloor$. Then $|S|^\epsilon<p+1\leq 2|S|^\epsilon<|S|^{2\epsilon}$. Set $d:=p+1$. Then, by the aforementioned result of~\cite{lubotzky}, we may choose a sequence of $d$-regular graphs as in Proposition~\ref{prop:small_cancellation}, with $A$ as above. Let $\nu$ be as obtained from the proposition.

By Proposition~\ref{prop:small_cancellation}, there exist arbitrarily large $n$ for which there exist labellings of $\Theta_n$ by $S$ satisfying the conclusions (i)--(iv). We choose one such labelling of a $\Theta_n$ for which $\nu\girth(\Theta_n)$ is large enough such that a ball of radius $\nu\girth(\Theta_n)$ in $\Gamma$ contains $K\cup S$. We show that $\Lambda:=\Gamma/\Theta_n$ satisfies the conclusion of Theorem~\ref{theorem:hyperbolic_quotient}. Following Proposition~\ref{prop:small_cancellation} (i) and (ii), all that remains to argue is that $\Lambda$ contains a subset $F$ such that $|E_{\pi(S)}(F)|\geq \frac{1}{2}|F||\pi(S)|^\epsilon$ so that we get $e(\Lambda,\pi(S)) \geq \frac12|S|^{-1+\epsilon'}$ as required.
 
\begin{lemma} \label{lemtheta}
Let $f(\Theta_n)$ be one of the images of $\Theta_n$  in $\Cay(\Gamma/\Theta_n,\pi(S))$. Then
$$\frac{|E(f(\Theta_n))|}{|V(f(\Theta_n))|}\geq\frac{|E(\Theta_n)|}{|V(\Theta_n)|}.$$
\end{lemma}
This lemma only requires conclusions (iii) and (iv) of Proposition~\ref{prop:small_cancellation}. It applies to any group $\Gamma$, any subset $S$ and any finite $S$-labelled graph $\Theta_n$ for which these two conclusions as well as $|E(\Theta_n)|\geq |V(\Theta_n)|$ hold.
\begin{proof}
In the notation of Proposition~\ref{prop:small_cancellation} (iv), let $H_n\backslash T_n=:\Omega$. For ease of notation, set $\Theta:=\Theta_n$, $T:=T_n$, and $H:=H_n$. Recall from Proposition~\ref{prop:small_cancellation} (iv) that $\Omega\cong f(\Theta)$. If $T$ is obtained by mapping $\tilde\Theta$ to $\Cay(\Gamma,S)$ by sending an element of the fiber $F$ of a base vertex $v$ in $\Theta$ to the identity in $\Gamma$, then $H\leq \Gamma$ is the set of elements of $\Gamma$ in the image of $F$. Thus $V(T)$ contains $H$ and, since $T$ is connected, $H$ is the subset of $\Gamma$ represented by the set $W$ of words read on paths in $T$ that connect elements of $H$. $W$ is also the set of words read on closed paths in $H\backslash T$ based at the trivial coset $H$. Thus, the image of $\pi_1(\Omega,H)$ in $\Gamma$ by the homomorphism induced by the labelling is $H$. By Proposition~\ref{prop:small_cancellation} (iii), $H$ is an isomorphic copy of $\pi_1(\Theta,v)$. Since $\pi_1(\Omega,H)$ surjects onto $H$, we have $\rank(\pi_1(\Omega,H))\geq\rank(\pi_1(\Theta,v))$.
Now we may express the rank of the fundamental group of a graph in terms of numbers of edges and vertices, thus obtaining:
$|E(\Omega)|-|V(\Omega)|+1\geq |E(\Theta)|-|V(\Theta)|+1$. Since $|V(\Omega)|\leq |V(\Theta)|$ and $|E(\Theta)|\geq |V(\Theta)|$ by construction, this yields:
\begin{equation*}
 \frac{|E(\Omega)|}{|V(\Omega)|}=\frac{|E(\Omega)|-|V(\Omega)|}{|V(\Omega)|}+1\geq\frac{|E(\Theta)|-|V(\Theta)|}{|V(\Omega)|}+1\geq \frac{|E(\Theta)|-|V(\Theta)|}{|V(\Theta)|}+1=\frac{|E(\Theta)|}{|V(\Theta)|}.
\end{equation*}
Thus, the proof is finished in view of the isomorphism $\Omega \cong f(\Theta)=f(\Theta_n)$.
\end{proof}

We are now ready to conclude the proof of Theorem~\ref{theorem:hyperbolic_quotient}: if $F:=V(f(\Theta_n))$, using the fact that $\Theta_n$ is $d$-regular and that $\pi$ is injective on $S$, we have using Lemma~\ref{lemtheta}:
$$\frac{|E_{\pi(S)}(F)|}{|F|}\geq \frac{|E(f(\Theta_n))|}{|V(f(\Theta_n))|}\geq \frac{|E(\Theta_n)|}{|V(\Theta_n)|}= \frac{d}{2}\geq \frac{1}{2}|S|^\epsilon=\frac12|\pi(S)|^{\epsilon}$$ and hence $e(\Lambda,\pi(S)) \geq \frac12|\pi(S)|^{-1+\epsilon}\geq \frac12|\pi(S)|^{-1+\epsilon'}$ as required. This finishes the proof of Theorem~\ref{theorem:hyperbolic_quotient}.
\end{proof}

\section{Forests, \texorpdfstring{$L^2$}{L2}-invariants and the proof of Theorem~\ref{thm:torsion}}
Recall that a \textbf{forest} on a group $\Gamma$ is a subset $F \se \Gamma \times \Gamma$ such that the resulting graph $(\Gamma,F)$ has no cycles. The collection $\mathfrak{F}_{\Gamma}$ of all forests on $\Gamma$ is a closed $\Gamma$-invariant subspace of the compact $\Gamma$-space of all subsets of $\Gamma\times \Gamma$ with respect to usual product topology. A \textbf{random forest} is a $\Gamma$-invariant Borel probability measure $\mu$ on $\mathfrak{F}_{\Gamma}$. The expected degree of a vertex in a random forest does not depend on the vertex; it is thus called the \textbf{expected degree of the forest}, denoted by $\deg(\mu)$. We further recall that the \textbf{width} of $\mu$ is the number ${\rm width}(\mu) \geq \deg(\mu)$ of vertices that neighbour a given vertex with positive probability.

The following is recorded in Proposition~2.3 of~\cite{EpsteinMonod} when $p=2$.

\begin{prop}\label{prop:forst:p}
Let $\mu$ be a random forest of finite width on $\Gamma$. Then
$$\|f_{\mu}\|_{T_1(\Gamma)} \leq 2 \kern3mm\text{and}\kern3mm \|f_{\mu}\|_p \geq \deg(\mu) \big({\rm width}(\mu)\big)^{-\frac{p-1}p}.$$
\end{prop}

\begin{proof}
The first inequality is unchanged from~\cite{EpsteinMonod}, and the second comes from replacing Cauchy--Schwarz by H\"older. The countability assumption in~\cite{EpsteinMonod} is not needed here.
\end{proof}

This change from~$2$ to~$p$ gives the following version of Theorem~1.5 of~\cite{EpsteinMonod}. We denote by $\beta^{(2)}_1(\Lambda)$ the first $L^2$-Betti number of a group $\Lambda$ and by ${\rm rk}(\Gamma)$ the minimal number of generators.

\begin{theorem}\label{thm:ibet:l2}
Let $\Gamma$ be a group and $\epsilon > \ibet(\Gamma)$. Then
$$\sup_{\Lambda} \frac{\beta^{(2)}_1(\Lambda)}{{\rm rk}(\Lambda)^\epsilon} < \infty,$$
where the supremum runs over all finitely generated subgroups $\Lambda$ of $\Gamma$.
\end{theorem}

\begin{proof}
Suppose that $T_1(\Gamma) \se \ell^p(\Gamma)$. Then Theorem~1.3 of~\cite{EpsteinMonod}, after replacing~$2$ by~$p$ in its proof using Proposition~\ref{prop:forst:p} above, states that
$$\frac{\deg(\mu)^{\frac p{p-1}}}{{\rm width}(\mu)}$$
remains bounded as $\mu$ ranges over all random forests of finite width defined on all countable subgroups $\Lambda$ of $\Gamma$. The rest of the proof is unchanged from~\cite{EpsteinMonod}, as follows. There is a particular random forest on $\Lambda$, the free uniform spanning forest, which is known to have width at most~$2\,{\rm rk}(\Lambda)$ and degree at least $2\,\beta^{(2)}_1(\Lambda)$. The statement now follows since we can take $p$ such that $(p-1)/p=\epsilon$ by Theorem~\ref{thm:Lit:ibet}.
\end{proof}

\begin{proof}[Proof of Corollary~\ref{bettili}]
We argue exactly as in~\cite{EpsteinMonod}: for any finite index subgroup $\Lambda<\Gamma$, one has
$$\beta^{(2)}_1(\Lambda) = [\Gamma:\Lambda]\, \beta^{(2)}_1(\Gamma)  \kern5mm\text{and}\kern5mm {\rm rk}(\Lambda) \leq  [\Gamma:\Lambda]\, {\rm rk}(\Gamma).$$
The equality above is a basic property of $L^2$-Betti numbers~\cite[1.35(9)]{luck}; the inequality is an elementary consequence of the Reidemeister--Schreier algorithm~\cite[Prop.~4.1]{Lyndon-Schupp}. Therefore, Theorem~\ref{thm:ibet:l2} shows that $\ibet(\Gamma)=0$ if $0<\beta^{(2)}_1(\Gamma) < \infty$ and if $\Gamma$ has subgroups of arbitrarily large finite index, which is the case for residually finite groups.
\end{proof}

\begin{proof}[Proof of Theorem~\ref{thm:torsion}]
This result is proved using Theorem~\ref{thm:ibet:l2} exactly as Osin used Theorem~1.5 of~\cite{EpsteinMonod} for his torsion non-unitarisable group in~\cite{Osin}. First, Theorem~2.3 in~\cite{Osin} gives a sequence $\Gamma_n$ of $n$-generated torsion groups such that $\beta^{(2)}_1(\Gamma_n) \geq n-2$. Next, by a result of Ol\cprime shanski\u\i~\cite{Olshan}, there is a simple $2$-generated torsion group $\Gamma$ containing $\bigoplus_n \Gamma_n$ and hence each $\Gamma_n$. If now we had $\Lit(\Gamma)<\infty$, then we would contradict Theorem~\ref{thm:ibet:l2} by choosing $\epsilon=(p-1)/p$ for $p$ with $\Lit(\Gamma)<p<\infty$.
\end{proof}

\section{Asymptotics of the spectral radius}
\subsection{Relation with \texorpdfstring{$\ibet(\Gamma)$}{eta}}
Considering the definitions of $r(\Gamma)$ and of $\ibet(\Gamma)$, the Cheeger inequalities in the form of Remark~\ref{rem:Mohar} imply the following.

\begin{prop}\label{betaandr}
For any infinite group $\Gamma$ we have $r(\Gamma)\leq\ibet(\Gamma) \leq 2 r(\Gamma)$.\qed
\end{prop}

Thus we see in hindsight that the proof of Theorem~\ref{beginning} for $\ibet(\Gamma)$ was really a statement about $r(\Gamma)$. We shall therefore investigate the latter invariant a bit further.

Here is the summary of what we know so far about $r(\Gamma)$.

\begin{prop}\label{prop:r}
Let $\Gamma$ be any infinite group.
\begin{enumerate}
\item If $\Gamma$ is amenable, then $r(\Gamma)=0$.
\item If $\Gamma$ is non-amenable, then $r(\Gamma)>0$.
\item For any group $\Gamma$ we have $r(\Gamma) \leq 1/2$\label{pt:half}.
\item If $\Gamma$ contains a non-abelian free subgroup, then $r(\Gamma)=1/2$ and $\ibet(\Gamma)=1.$
\end{enumerate}
\end{prop}


\begin{proof}
Combining Theorem~\ref{beginning} and Theorem~\ref{thm:Lit:ibet}, we have $\ibet(\Gamma)=0$ if and only if $\Gamma$ is amenable. Thus~(1) and~(2) follow from Proposition~\ref{betaandr}.

Next, we recall that $\rho(\Gamma, S)\geq (|S|-1)^{1/2} /|S|$ holds any symmetric finite set $S$ in any group $\Gamma$ (this goes back to Kesten~\cite{Kesten}). This implies~(3) since $|S|\to\infty$. Finally, for~(4), we recall from Theorem~\ref{thm:Lit:ibet} that $\ibet(\Gamma)=1$ if $\Gamma$ contains a non-abelian free subgroup; thus $r(\Gamma)\geq 1/2$ in that case, by Proposition~\ref{betaandr}.
\end{proof}

\begin{note}
Let us note that $r(\Gamma)=1/2$ does not characterise the existence of free subgroup. For example, if we again consider the group constructed by Osin in~\cite{Osin}, we will have the following: we know that $\Lit(\Gamma)=\infty$ and hence $\ibet(\Gamma)=1$. So we have $1=\ibet(\Gamma) \leq 2r(\Gamma)$ or equivalently $r(\Gamma) \geq 1/2$. Therefore $r(\Gamma)=1/2$ holds in view of point~\eqref{pt:half} in Proposition~\ref{prop:r}.
\end{note}

\subsection{Behaviour under quotients}
The definition of $r$ implies $r(\Lambda)\leq r(\Gamma)$ when $\Lambda$ is a subgroup of $\Gamma$. Just like for $\Lit$ (or equivalently $\ibet$), we also have monotonicity for quotients:

\begin{lemma}\label{lem:rsurjection}
If $\pi \colon \Gamma\to \Lambda $ is an epimorphism, then $r(\Gamma) \geq r(\Lambda)$ holds.
\end{lemma}

\begin{proof}
Given a finite symmetric set $\Sigma \se \Lambda$, any symmetric set $S \se \Gamma$ that is mapped $1$-to-$1$ onto $\Sigma$ by $\pi$ satisfies $\rho(\Gamma, S) \leq \rho(\Lambda, \Sigma)$. This implies the statement in the case where we can always find such a set $S$. However a potential obstruction to the symmetry of $S$ arises in case $\pi$ creates new $2$-torsion.

In the general case, we can assume $\ker(\pi)$ non-trivial and hence we can always find a symmetric set $S \se \Gamma$ that is mapped $2$-to-$1$ onto $\Sigma$. We have again $\rho(\Gamma, S) \leq \rho(\Lambda, \Sigma)$ and now 
$$\frac{- \ln \rho(\Gamma, S) }{\ln |S|} \geq  \frac{-\ln \rho(\Gamma, \Sigma) }{\ln 2 + \ln |\Sigma|}$$
implies the statement since the size of $\Sigma$ goes to infinity.
\end{proof}

Using a classical result of Kesten~\cite{Kesten} and a result from~\cite{Thom:rad}, we shall establish equality for amenable kernels.

\begin{prop}\label{rsurjection:eq}
If the kernel of the epimorphism $\pi \colon \Gamma\to \Lambda $ is amenable, then $r(\Gamma) = r(\Lambda)$ holds.
\end{prop}

Contrary to the case of Kesten's statement, we do not see why the above proposition should admit a converse.

\begin{proof}[Proof of Proposition~\ref{rsurjection:eq}]
Fix $S \se \Gamma$ a finite symmetric subset. Consider the multiplicity function $a \colon \Lambda \to \NN$ defined by $a(g) = | \{ h \in S : \pi(h)=g\}|$. Applying Corollary~4 of~\cite{Thom:rad} on $\Lambda$, there exists an integer $k \geq 1$ and a symmetric finite set $\Sigma \se \Lambda$ such that $a(g) \geq k$ for all $g \in \Sigma$ and satisfying
\begin{equation}\label{eq:onestep}
k  |\Sigma| \geq \frac{|S|}{4 \ln|S|}.
\end{equation}
We can thus choose a subset of $S$ which is mapped $k$-to-$1$ onto $\Sigma$ by $\pi$. We want, however, a symmetric set; the issue is the same as in the proof of Lemma~\ref{lem:rsurjection} but with the additional constraint that we will need a subset of $S$. We can indeed choose a symmetric $S' \se S$ if we only require that $\pi$ maps $S'$ onto $\Sigma$ with each fibre containing either $k$ or $k + 1$ elements. We write $\Sigma = \Sigma_0 \sqcup \Sigma_1$ for the corresponding partition of $\Sigma$, noting that both $\Sigma_i$ can be chosen symmetric since we determine them according to $2$-torsion properties. Since $\ker(\pi)$ is amenable, Corollary~2 in~\cite{Kesten} states that $\rho(\Gamma,S')$ coincides with the spectral radius of the Markov operator $\pi(M_{S'})$ on $\Lambda$. Explicitly, we have
$$\pi(M_{S'}) = c \Big( k \one_{\Sigma_0} + (k+1) \one_{\Sigma_1}\Big),$$
where $c$ is the normalization constant ensuring $c (k |\Sigma_0| + (k+1) |\Sigma_1|)=1$. Therefore $M_\Sigma\leq \frac{k+1}{k} \pi(M_{S'}) \leq 2 \pi(M_{S'})$ pointwise and hence, using the monotonicity of the spectral radius, we deduce $\rho(\Lambda,\Sigma) \leq 2 \rho(\Gamma,S')$. Moreover, $\rho(\Gamma,S') \leq \rho(\Gamma,S) \cdot |S|/|S'|$ holds also by monotonicity because $S'\se S$. On the other hand, combining $k  |\Sigma| \leq |S'|$ with the estimate~\eqref{eq:onestep} above, we have $|S|/|S'| \leq 4 \ln |S|$. In summary, we have
$$\rho(\Lambda,\Sigma) \leq 8 \ln|S| \cdot \rho(\Gamma,S).$$
Together with the trivial estimate $|\Sigma| \leq |S|$, we can conclude
$$\frac{- \ln \rho(\Lambda,\Sigma)}{\ln |\Sigma|} \geq \frac{- \ln \rho(\Gamma,S) - \ln \ln |S| - \ln 8}{\ln |S|},$$
which yields $r(\Lambda) \geq r(\Gamma)$ since $|S|$ goes to infinity. This completes the proof in view of Lemma~\ref{lem:rsurjection}.
\end{proof}

\subsection{An alternative expression for \texorpdfstring{$r(\Gamma)$}{r}}
Finally, we record that we can replace the limes inferior by an infimum in the definition of $r(\Gamma)$.

\begin{theorem}\label{rinf}
For any group $\Gamma$ we have $\displaystyle r(\Gamma)=-\inf_S \frac{\ln \rho(\Gamma,S)}{\ln|S|}$.
\end{theorem}

This statement is a formal consequence of the following result of one of the authors:

\begin{theorem}[\cite{Thom:rad}]\label{thm:technical}
Let $\Sigma$ be a finite symmetric subset of a group $\Gamma$. For any $0< \epsilon< -\frac{\ln \rho(\Gamma,\Sigma)}{\ln|\Sigma|}$ there exists a sequence $(S_k)$ of finite symmetric subsets $S_k\se \Gamma$ whose size tends to infinity and such that $\rho(\Gamma,S_k) \leq |S_k|^{-\epsilon}$ holds for all $k$.
\end{theorem}

\begin{proof}
Note that the statement is empty unless $\Sigma$ generates a non-amenable group. The proof of Corollary~6 in~\cite{Thom:rad} gives a sequence $(S_k)$ such that the desired inequality holds for all large $k$ when $\epsilon = \frac 12 (-\frac{\ln \rho(\Gamma,\Sigma)}{\ln|\Sigma|})$. But this proof never used more than  $\epsilon< -\frac{\ln \rho(\Gamma,\Sigma)}{\ln|\Sigma|}$; varying $\epsilon$ will only change the first index $k$ beyond which the estimate holds, and hence change how to truncate the sequence if we want the estimate for \emph{all} $k$.
\end{proof}

\section{Applications and open problems}
\subsection{Rapid decay property}
Let $1\leq p <\infty$. Recall that a group $\Gamma$ has the \textbf{$p$-rapid decay property RD$_p$} if there exists a length function $L$ on $\Gamma$ and a polynomial $P$ such that
\begin{equation}\label{eq:RD}
\|a\|_{p \to p} \leq P(d)\, \| a \|_p
\end{equation}
is satisfied for every $a\in {\mathbf C}[\Gamma]$ and ${d\geq 0}$ such that $a$ is supported in an $L$-ball of radius~$d$. Note that when $\Gamma$ is finitely generated, it is equivalent to require~\eqref{eq:RD} for some (or any) word-length~$L$. We refer to~\cite{liao-yu} for background and recall the following:

\begin{itemize}
\item For $p=2$, RD$_2$ is the classical rapid decay property RD introduced by Jolissaint~\cite{joli}.
\item RD$_1$ always holds and RD$_p$ is equivalent to polynomial growth if $p>2$.
\item RD$_p$ implies RD$_q$ if $p>q$.
\end{itemize}

In other words, as $p$ decreases from~$2$ to~$1$, property~RD$_p$ is weakening of~RD until no restriction is left.

\begin{prop}\label{prop:RD}
Let $\Gamma$ be a non-amenable group.

If $\Gamma$ has property RD$_p$ for $p \in [1,2]$, then $r(\Gamma)\geq 1-1/p$.

In particular, if $\Gamma$ has property RD, then $r(\Gamma)=1/2$.
 \end{prop}

In view of Corollary~\ref{chain}, Proposition~\ref{prop:RD} implies the following.

 \begin{cor}
Let $\Gamma$ be a non-amenable group with property RD$_p$ for $p\in[1,2]$.

Then $\Lit(\Gamma) \geq p$.\qed
 \end{cor}

\begin{proof}[Proof of Proposition~\ref{prop:RD}]
Fix a length $L$ on $\Gamma$ satisfying the inequality~\eqref{eq:RD} and consider any symmetric finite set $S\se \Gamma$. Then the Markov operator $M_S$ on $\ell^p(\Gamma)$ has norm at most
$$\|M_S\|_{p \to p} \leq P(d)\, \| M_S \|_p =   P(d)\, |S|^{\frac{1-p}{p}},$$
where $d$ is the radius of a ball containing $S$. Let $q$ be the conjugate exponent of $p$. Since $S$ is symmetric, the adjoint on $\ell^q(\Gamma)$ of the Markov operator $M_S$ on $\ell^p(\Gamma)$ is also given by $M_S$. It follows that $\|M_S\|_{q \to q}  \leq P(d)\, |S|^{(1-p)/p}$. Since~$2$ is the harmonic mean of~$p$ and~$q$, the Riesz--Thorin theorem yields $\|M_S\|_{2 \to 2} \leq P(d)\, |S|^{(1-p)/p}$.

Since $\Gamma$ is non-amenable, it contains a finitely generated subgroup of exponential growth. That is, there is $S_1\se \Gamma$ finite symmetric and $\omega>1$ with
$$\omega^d\leq |S_1^d| \leq |S_1|^d \kern5mm\forall\, d.$$
Upon adjusting $L$ (and hence $P$) by a constant, $S_1^d$ lies in an $L$-ball of radius $d$ for each $d\in\mathbf N$. Now we have
$$\rho(\Gamma, S_1^d) = \|M_{S_1^d}\|_{2 \to 2} \leq P(d) \, |S_1|^{d \frac{1-p}p}$$
and we can conclude
$$ \frac{-\ln \rho(\Gamma, S_1^d) }{\ln |S_1^d|} \geq \frac{d\, \frac{p-1}p \ln|S_1|}{d\ln |S_1|} - \frac{\ln P(d)}{d \ln \omega} = \frac{p-1}{p} -  \frac{\ln P(d)^{\frac1d}}{\ln \omega}.$$
Since $P$ is a polynomial, the limit as $d\to\infty$ of the right hand side is $(p-1)/p=1-1/p$, as required. The additional statement follows since we always have $r(\Gamma)\leq 1/2$.
\end{proof}

\subsection{Burnside groups}
We first recall the definition of the \textbf{cogrowth} $\alpha$ associated to a finitely generated group $\Gamma$ endowed with a choice of~$m$ generators. This choice corresponds to an epimorphism $\pi\colon F_m\to\Gamma$ from the free group $F_m$ of rank $m$, which we endow with the word-length $L$ associated to the chosen basis. Then the definition is
$$\alpha=\limsup_{k \to \infty} \big|\{w\in \ker \pi : L(w)=k \}\big|^\frac1k.$$
We assume that $\pi$ is not the identity; one then has
\begin{equation}\label{eq:alpha}
\sqrt{2m -1} \leq \alpha \leq 2m -1.
\end{equation}
The first inequality is proved in Statement~3.1 of~\cite{grigorchuk} and the second holds by definition. Let now $S\se \Gamma$ be the set consisting of the chosen generators together with their inverses. Then Theorem~4.1 in~\cite{grigorchuk} states
$$\rho(\Gamma,S)= \frac{\sqrt{2m-1}}{2m}\left(\frac{\sqrt{2m-1}}{\alpha}+\frac{\alpha}{\sqrt{2m-1}}\right),$$
which, with the lower bound of~\eqref{eq:alpha} for $\alpha$, implies
\begin{equation}\label{eq:cogrowth}
\rho(\Gamma,S) \leq \alpha /m.
\end{equation}
We now turn to the free Burnside group $B(m, a)$ of exponent $a$ on $m$ generators. As explained in~\cite[\S60]{TGH} (see also Remark~\ref{remark:better_burnside_bounds} below), Adyan proved in~\cite{Adyan} that there is $\delta<2/3$ such that the estimate
$$\alpha \leq (2m-1)^{\delta}$$
holds for $m\geq 2$ and $a\geq 665$ odd. Therefore, using~\eqref{eq:cogrowth}, we have for the corresponding symmetric set $S$ the estimate
\begin{equation}\label{eq:Burn}
\rho(B(m, a), S) \leq (2m-1)^{-\frac13}
\end{equation}
for all $m$ large enough. We are now ready to deduce the following.

\begin{theorem}\label{thm:Burn}
Let $B(m,a)$ be the free Burnside group of exponent $a$ on $m$ generators, where $m \geq 2$ and $a \geq 665$ odd.

Then $r(B(m,a))\geq 1/3$ and hence $\Lit(B(m,a)) \geq 3/2$.
\end{theorem}

\begin{proof}
The second statement follows from the first by Corollary~\ref{chain}. For the first statement, it is sufficient to establish the case $m=2$. Indeed, the universal property implies that $B(2,a)$ is a quotient of $B(m,a)$ and therefore we can apply Lemma~\ref{lem:rsurjection} to reduce ourselves to $B(2,a)$.

We shall nonetheless use the groups $B(m,a)$, as follows. It was proved by \v{S}irvanjan~\cite{Sirvajan} that $B(m,a)$ embeds into $B(2,a)$ (still under the assumption $a \geq 665$ odd). Therefore, we obtain a sequence $(S_m)$ of symmetric sets $S_m\se B(2,a)$ for which the estimate~\eqref{eq:Burn} implies
$$\rho(B(2, a), S_m) \leq (2m-1)^{-\frac13}.$$
We deduce
$$r(B(2, a)) \geq \limsup_{m\to\infty} \frac{-\ln \rho(B(2, a), S_m) }{\ln |S_m|}  \geq \frac13 \limsup_{m\to\infty}\frac{\ln (2m-1)}{\ln 2m} = \frac13,$$
as desired.
\end{proof}

\begin{note} \label{remark:better_burnside_bounds}
Given $m\geq 2$, Adyan~\cite{Adyan} provides upper bounds for $\alpha$ that converge to $(2m-1)^{1/2+2/q}$ as the odd exponent $a$ goes to infinity. More precisely, these estimates follow from Theorem~3 in~\cite{Adyan} (there is a misprint in this translation, the exact formula is in the corresponding theorem in the original~\cite{Adyan_r}). In the notation of that formula, we have $\gamma_R\to 1$ and $\delta_R\to\infty$ as $a\to\infty$ (odd), and the claim follows.

Here $q=90$ is a fixed parameter from~\cite{Adyan-book}. It is mentioned in~\cite{Adyan-book}*{VI.2.16 (page 254)} that the value of $q$ can be increased at the cost of increasing also the exponent $a$. Hence, increasing the parameter $q$ and $a$ accordingly, the proof of Theorem~\ref{thm:Burn} would show that there are free Burnside groups $B(m,a)$ with $r(B(m,a))$ bounded below by values arbitrarily close to~$1/2$ and, consequently, $\Lit(B(m,a))$ bounded below by values arbitrarily close to~$2$. 
\end{note}

It was proven by Ozawa and one of the authors~\cite{MonodOzawa} that the groups $\Gamma=B(m,na)$ are non-unitarisable for $m,n \geq 2$, $a \geq 665$, $n, a$ odd. This may seem to indicate that $\Lit(\Gamma) \geq 2$. In view of the loss of a factor of two in the comparison between $\ibet(\Gamma)$ and $r(\Gamma)$, one might even speculate that $\Lit(\Gamma) \geq 3$.

\subsection{Colourings}
We define the \textbf{maximal average degree} of a graph $G$ by
$${\rm mad}(G)=\sup_{F}\frac{2|E(F)| - |L(F)|}{|F|}, $$
where $F$ runs through all finite subsets of vertices. In the case of Cayley graph ${\rm Cay}(\Gamma,S)$, we have
\begin{equation}\label{eq:mad}
{\rm mad}({\rm Cay}(\Gamma,S))= |S| \,e(\Gamma, S)
\end{equation}
by Lemma~\ref{lem:count:E}. On the other hand, there is a well-known relation between ${\rm mad}(G)$ and colourings:

\begin{prop}\label{prop:mad}
Let $G$ be any locally finite graph without loops. If ${\rm mad}(G)\leq k$, then $G$ is $(k+1)$-colourable.
\end{prop}

\begin{proof}
By a compactness argument, it is enough to prove the claim for finite graphs. Suppose by induction that we have proved the statement for any graph with $n$ vertices and let $G$ be a graph with $n+1$ vertices. Since ${\rm mad}(G)\leq k$, we can find a vertex $v$ of degree~$\leq k$. Now, the graph $G \setminus \{v\}$ is $(k+1)$-colourable by induction and there is one colour left for $v$.
\end{proof}

We can now deduce the

\begin{proof}[Proof of Corollary~\ref{colourcor}]
In view of Theorem~\ref{thm:Lit:ibet}, it is equivalent to prove that for every $\eta'<\ibet(\Gamma)$ there are arbitrarily large finite symmetric sets $S$ for which ${\rm Cay}(\Gamma,S)$ is $|S|^{1-\eta'}$-colourable.  Choose $\eta'<\eta''<\ibet(\Gamma)$. By definition of $\ibet(\Gamma)$, there are arbitrarily large $S$ with $e(\Gamma, S) < |S|^{-\eta''}$. Since $|S|$ goes to infinity, there is no loss of generality in possibly removing the identity from $S$ to avoid loops. By~\eqref{eq:mad}, we have ${\rm mad}({\rm Cay}(\Gamma,S)) < |S|^{1-\eta''}$ and hence Proposition~\ref{prop:mad} implies that ${\rm Cay}(\Gamma,S)$ is $(|S|^{1-\eta''}+1)$-colourable. Since $|S|$ goes to infinity and $\eta'<\eta''$, we can assume that $|S|^{1-\eta''}+1$ is less than $|S|^{1-\eta'}-1$ and the proof is complete.
\end{proof}

\subsection{Open problems}
It would be very desirable to find more examples of groups $\Gamma$ with $1<\Lit(\Gamma)<\infty$. Perhaps this would give some insight into the following:

\begin{question}\label{q:range}
Is there any restriction on the value of the Littlewood exponent?
\end{question}

Of particular interest is the following specification of Question~\ref{q:range}.

\begin{question}\label{q:forbid}
Is the interval $(1,2]$ a forbidden range for the Littlewood exponent?
\end{question}

Recall that unitarisable groups satisfy $\Lit(\Gamma)\leq 2$. Therefore, in view of Theorem~\ref{beginning}, Dixmier's problem is equivalent to the conjunction of Question~\ref{q:forbid} with the following.

\begin{question}\label{q:leq2}
Does the condition $\Lit(\Gamma)\leq 2$, or equivalently $\ibet(\Gamma) \leq 1/2$, characterise unitarisability?
\end{question}

In other words, Dixmier's problem is decomposed into two apparently independent open problems: Question~\ref{q:forbid} and Question~\ref{q:leq2}.

There is however a form of overlap between these two questions, as can be shown using the main result of~\cite{MonodOzawa}, as follows.

\begin{cor}\label{cor:wreath}
If $\Lit(\Gamma)\leq 2$ characterises unitarisability, then  $(1,4/3]$ is a forbidden range for the Littlewood exponent.
\end{cor}

\begin{proof}
In view of Theorem~\ref{thm:Lit:ibet}, we need to prove that $\ibet(\Gamma)\leq 1/4$ implies that $\Gamma$ is amenable. By Corollary~\ref{chain}, we have $r(\Gamma)\leq 1/4$. Proposition~\ref{rsurjection:eq} implies that the wreath product $\mathbf Z \wr \Gamma$ also satisfies $r(\mathbf Z \wr \Gamma)\leq 1/4$. Appealing again to Corollary~\ref{chain}, we have $\ibet(\mathbf Z \wr \Gamma)\leq 1/ 2$. Our assumption now implies that $\mathbf Z \wr \Gamma$ is unitarisable. It was proved in~\cite{MonodOzawa} that this implies that $\Gamma$ is amenable.
\end{proof}

\begin{question}
Let $B(m,a)$ be the free Burnside group of exponent $a$ on $m$ generators, where $m \geq 2$ and $a \geq 665$ odd. Can one give more precise bounds for $\Lit(B(m,a))$? Do we have $\Lit(B(m,a))=\infty$?
\end{question}

It is unknown whether unitarisability is preserved under direct products of groups. It is easily seen to be preserved under extensions with amenable quotients, but unlikely to be preserved under extensions with amenable kernels due to the main result of~\cite{MonodOzawa} cited above.

\begin{question}
How does the Littlewood exponent behave with respect to direct products of groups? How about extensions with amenable quotients or amenable kernels?
\end{question}

If $\Lit$, or equivalently $\ibet$, were preserved under extensions with amenable kernels in analogy to Proposition~\ref{rsurjection:eq}, then the same argument as in Corollary~\ref{cor:wreath} above would show that Question~\ref{q:leq2} implies Question~\ref{q:forbid} and hence becomes equivalent to Dixmier's problem.

\smallskip

Finally, we recall that significant progress on the Dixmier problem was obtained by Pisier~\cite{Pis98}, who introduced an exponent measuring the cost of unitarising a given representation. Using the theory of operator spaces, Pisier proved that his exponent takes only half-integer values and that the lowest value characterises amenability.

\begin{question}
Is there a relation between Pisier's exponent and the Littlewood exponent?
\end{question}

\section*{Acknowledgements}
We are grateful to R\'emi Coulon for suggesting that better bounds on $\Lit(B(m,a))$ should be available when $a$ goes to infinity, and we thank Denis Osin for two helpful comments on the first version.

This research was supported in part by the ERC Consolidator Grant No.\ 681207. The results presented in this paper are part of the PhD project of the first author.


\end{document}